\documentclass{article}

\usepackage{authblk}

\usepackage[utf8]{inputenc}
\usepackage[T1]{fontenc}
\usepackage[english]{babel}
\usepackage{textcomp}
\usepackage{amsmath,amssymb,amsthm,mathtools}
\usepackage{lmodern}
\usepackage[a4paper]{geometry}
\usepackage{xcolor, pict2e}
\usepackage{microtype}
\usepackage{listings}
\usepackage{multicol}
\usepackage{moreverb}
\usepackage{hyperref}
\hypersetup{pdfstartview=XYZ}
\usepackage{wrapfig}
\usepackage[sans]{dsfont}

\usepackage{tikz, pgfplots}
\usetikzlibrary{positioning}
\usetikzlibrary{decorations.pathreplacing}

\usepackage{ytableau}
\usepackage{stmaryrd}
\usepackage{mathdots}

\usepackage{soul}
\hypersetup{
    colorlinks=true,
    linkcolor=blue,
    citecolor=magenta,
    urlcolor=blue,
    pdfborder={0 0 0}
}

\DeclareMathOperator{\dtv}{d_{TV}}

\newcommand{\AddColors}{\textup{AddColor}}

\usepackage[font=sf, labelfont={sf,bf}, margin=1cm]{caption}
\usepackage{enumitem}
\setlist[itemize,1]{nosep}
\setlist[enumerate,1]{itemsep=0pt,label=(\alph*)}

\usepackage{float}
\usepackage[caption = false]{subfig}
\usepackage{graphicx}
\usepackage{bm}

\newtheorem*{theorem*}{Theorem}
\newtheorem{theorem}{Theorem}[section]

\newtheorem{lemma}[theorem]{Lemma}

\theoremstyle{remark}
\newtheorem{remark}{Remark}[section]

\usepackage{todonotes}

\newcommand{\old}[1]{}

\newcommand{\cC}{{\ensuremath{\mathcal C}} }

\newcommand{\cP}{{\ensuremath{\mathcal P}} }

\newcommand{\bbE}{{\ensuremath{\mathbb E}} }

\newcommand{\bbP}{{\ensuremath{\mathbb P}} }

\newcommand{\ag}{\left\{ } 

\newcommand{\ad}{\right\} }

\newcommand{\cg}{\left[}
\newcommand{\cd}{\right]}

\newcommand{\lf}{\left\lfloor}
\newcommand{\rf}{\right\rfloor}

\newcommand{\dt}{\mathrm d t}

\newcommand{\du}{{\ensuremath{\;:\;}}} 

\makeatletter
\newcommand*\bigcdot{\mathpalette\bigcdot@{.5}}
\newcommand*\bigcdot@[2]{\mathbin{\vcenter{\hbox{\scalebox{#2}{$\m@th#1\bullet$}}}}}
\makeatother

\numberwithin{equation}{section}

\pgfplotsset{compat=1.18}

\begin{document}

\renewcommand{\theparagraph}{\thesubsection.\arabic{paragraph}} 
\title{The Aldous chain on cladograms mixes in order $n^2$ steps}

\author{Valentin Féray}
\author{Lucas Teyssier}
\affil{Université de Lorraine, CNRS, IECL, F-54000 Nancy,\\ \texttt{$\{$valentin.feray,lucas.teyssier$\}$@univ-lorraine.fr}}
\date{}
\maketitle

\begin{abstract}
    Cladograms of size $n$ are unrooted binary trees whose leaves are labelled from 1 to $n$. Aldous introduced in 2000 a Markov chain on cladograms, 
    a step of which consists of removing a leaf uniformly at random and reinserting it on a uniformly chosen edge. 
We introduce a coupling for this walk,  which follows multiple colored subtrees in parallel. We go around the lack of independence of the colored components by finding a relevant statistic, namely the sum of the squares of the sizes of all colored components, which we prove has a drift. We deduce that the mixing time of the walk is of order $n^2$, solving a conjecture of Aldous. 
\end{abstract}


\section{Introduction}

\subsection{Background}

The mixing time of a Markov chain is the time needed for the distribution of the chain to be close to stationarity. 
Mixing times have been widely studied since the work of Diaconis and Aldous in the 1980s (see for instance \cite{DiaconisShahshahani1981, Aldous1983mixing, AldousDiaconis1986}). 
We refer to the textbooks \cite{LivreLevinPeres2019MarkovChainsAndMixingTimesSecondEdition, LivreDiaconisFulman2023} and the recent lecture notes \cite{Salez2025ModernAspectsOfMarkovChainsSaintFlour} for more history and results on mixing times.

In general, determining the mixing time of a Markov chain is a difficult problem, and sometimes, even
its order of magnitude is unknown.
This is the case of Aldous' walk on cladograms~\cite{Aldous2000MixingCladograms}. 
This Markov chain was introduced in 2000 by Aldous, who proved that the order of magnitude of the mixing time is between $n^2$ and $n^3$ and conjectured that the correct order is $n^2$ (see \cite[below Theorem 1]{Aldous2000MixingCladograms}).

After Aldous' work, the relaxation time
of his Markov chain was proved to be of order $n^2$ 
by Schweinsberg~\cite{Schweinsberg2002RelaxationTimeCladograms}, and the continuum limit of the model was recently understood \cite{LohrMytnikWinter2020CladogramsDiffusionLimit,forman2023Aldous-diffusion}.
This answered two other conjectures of Aldous, but no progress was made on the mixing time.
In this paper, we establish an upper bound of order $n^2$ on the mixing time, confirming Aldous' conjecture.

\subsection{The model}\label{s: the model}
Let $n\geq 3$, and let $\cC_n$ be the set of cladograms of size $n$, that is, of binary trees (unrooted, and not assumed to be planar) with $n$ leaves labelled from 1 to $n$, as illustrated in Figure~\ref{fig: cladogram example 20 uncolored}. Following Aldous \cite{Aldous2000MixingCladograms}, we consider
the following Markov chain on $\cC_n$.
Each step is decomposed into two half-steps that we now describe. Let $X \in \cC_n$.
The following construction is illustrated in Figure~\ref{fig: illustration of one step of the cladogram chain}.

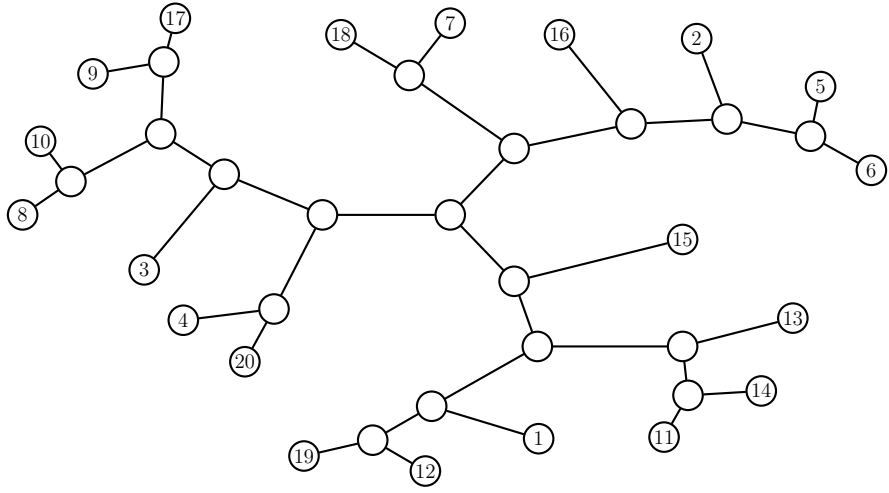
\begin{figure}[tbp]
\centering

\begin{tikzpicture}[
    scale=0.65,
    transform shape,
    every node/.style={
        circle,
        draw,
        minimum size=6mm,
        inner sep=0pt,
        font=\large
    },
    thick
]

\def\radiusone{2.60}
\def\radiustwo{4.80}
\def\radiusthree{6.50}
\def\radiusfour{7.80}
\def\radiusfive{8.70}

\def\layoutrotation{-90}
\def\verticalcompression{0.6}

\newcommand{\radialposition}[2]{%
    ({#2*cos(#1+\layoutrotation)},%
     {\verticalcompression*#2*sin(#1+\layoutrotation)})%
}

\node (C) at (0,0) {};


\node (A0)
    at \radialposition{150}{\radiusone} {};

\node (A1)
    at \radialposition{190}{\radiustwo} {};
\node (A2)
    at \radialposition{130}{\radiustwo} {};

\node (L1)
    at \radialposition{200}{\radiusthree} {$18$};
\node (L2)
    at \radialposition{180}{\radiusthree} {$7$};
\node (L3)
    at \radialposition{160}{\radiusthree} {$16$};
\node (A3)
    at \radialposition{120}{\radiusthree} {};

\node (L4)
    at \radialposition{140}{\radiusfour} {$2$};
\node (A4)
    at \radialposition{110}{\radiusfour} {};

\node (L5)
    at \radialposition{120}{\radiusfive} {$5$};
\node (L6)
    at \radialposition{100}{\radiusfive} {$6$};


\node (B0)
    at \radialposition{30}{\radiusone} {};

\node (L7)
    at \radialposition{80}{\radiustwo} {$15$};
\node (B1)
    at \radialposition{21.667}{\radiustwo} {};

\node (B2)
    at \radialposition{46.667}{\radiusthree} {};
\node (B4)
    at \radialposition{-3.333}{\radiusthree} {};

\node (L8)
    at \radialposition{63.333}{\radiusfour} {$13$};
\node (B3)
    at \radialposition{38.333}{\radiusfour} {};
\node (L11)
    at \radialposition{13.333}{\radiusfour} {$1$};
\node (B5)
    at \radialposition{-11.667}{\radiusfour} {};

\node (L9)
    at \radialposition{46.667}{\radiusfive} {$14$};
\node (L10)
    at \radialposition{30}{\radiusfive} {$11$};
\node (L12)
    at \radialposition{-3.333}{\radiusfive} {$12$};
\node (L13)
    at \radialposition{-20}{\radiusfive} {$19$};


\node (D0)
    at \radialposition{270}{\radiusone} {};

\node (D1)
    at \radialposition{311.667}{\radiustwo} {};
\node (D2)
    at \radialposition{253.333}{\radiustwo} {};

\node (L14)
    at \radialposition{320}{\radiusthree} {$20$};
\node (L15)
    at \radialposition{303.333}{\radiusthree} {$4$};
\node (L16)
    at \radialposition{286.667}{\radiusthree} {$3$};
\node (D3)
    at \radialposition{245}{\radiusthree} {};

\node (D4)
    at \radialposition{261.667}{\radiusfour} {};
\node (D5)
    at \radialposition{228.333}{\radiusfour} {};

\node (L17)
    at \radialposition{270}{\radiusfive} {$8$};
\node (L18)
    at \radialposition{253.333}{\radiusfive} {$10$};
\node (L19)
    at \radialposition{236.667}{\radiusfive} {$9$};
\node (L20)
    at \radialposition{220}{\radiusfive} {$17$};


\foreach \u/\v in {
    C/A0,
    A0/A1,A0/A2,
    A1/L1,A1/L2,
    A2/L3,A2/A3,
    A3/L4,A3/A4,
    A4/L5,A4/L6,
    C/B0,
    B0/L7,B0/B1,
    B1/B2,B1/B4,
    B2/L8,B2/B3,
    B3/L9,B3/L10,
    B4/L11,B4/B5,
    B5/L12,B5/L13,
    C/D0,
    D0/D1,D0/D2,
    D1/L14,D1/L15,
    D2/L16,D2/D3,
    D3/D4,D3/D5,
    D4/L17,D4/L18,
    D5/L19,D5/L20}
{
    \draw (\u)--(\v);
}

\end{tikzpicture}

\caption{A cladogram with 20 leaves.}
\label{fig: cladogram example 20 uncolored}

\end{figure}

\paragraph*{First half-step.} We pick a leaf $\ell$ uniformly at random in $X$. Denote the neighbour of $\ell$ by $u$, and the neighbours of $u$ by $u_1$ and $u_2$. Then we replace the two edges $\ag u,u_1 \ad$ and $\ag u, u_2 \ad$ by a single edge $\ag u_1, u_2\ad$. This gives a cladogram $X'$ with $n-1$ leaves (labeled by the integers from $1$ to $n$, except $\ell$), and an additional component consisting of $u, \ell$ and $\ag u, \ell \ad$.

\paragraph*{Second half-step.} We pick an edge $\ag v_1, v_2 \ad$ of $X'$ uniformly at random (picking $\ag u_1, u_2\ad$ is a possibility), and replace it by two edges $\ag v_1, u\ad$ and $\ag v_2, u\ad$.
Informally, the leaf $\ell$ is attached with its incident edge and vertex on $\ag v_1, v_2 \ad$.
This yields a cladogram $X''$ in $\cC_n$, which is by definition the output of a single step of Aldous' Markov chain applied to $X$.\medskip

\begin{figure}[ht]
\centering

\begin{tikzpicture}[  scale=0.6,  transform shape,
    every node/.style={
        circle,
        draw,
        minimum size=6mm,
        inner sep=0pt
    },
    thick
]

\node (a) at (-2,-1) {};
\node (b) at (-1,0) {};
\node (d) at (-2,1) {};
\node (e) at (-3,2) {};
\node (f) at (-1,2) {};

\node (L1) at (-1,-2) {$1$};
\node (L2) at (0,-1) {$2$};
\node (L3) at (0,3) {$3$};
\node (L4) at (0,1) {$4$};
\node (L5) at (-3,-2) {$5$};
\node (L6) at (-4,1) {$6$};
\node (L7) at (-4,3) {$7$};

\draw (L1)--(a);
\draw (a)--(b);
\draw (a)--(L5);

\draw (b)--(L2);
\draw (b)--(d);

\draw (d)--(e);
\draw (d)--(f);


\draw (e)--(L6);
\draw (e)--(L7);

\draw (f)--(L3);
\draw (f)--(L4);
\end{tikzpicture}
\hspace{3em}
\begin{tikzpicture}[  scale=0.6,  transform shape,
    every node/.style={
        circle,
        draw,
        minimum size=6mm,
        inner sep=0pt
    },
    thick
]

\node[gray] (a) at (-3,-1) {};
\node (b) at (-1,0) {};
\node (d) at (-2,1) {};
\node (e) at (-3,2) {};
\node (f) at (-1,2) {};

\node (L1) at (-2,-1) {$1$};
\node (L2) at (0,-1) {$2$};
\node (L3) at (0,3) {$3$};
\node (L4) at (0,1) {$4$};
\node[gray] (L5) at (-4,-2) {$5$};
\node (L6) at (-4,1) {$6$};
\node (L7) at (-4,3) {$7$};

\draw (L1)--(b);
\draw[gray] (a)--(L5);

\draw (b)--(L2);
\draw (b)--(d);

\draw (d)--(e);
\draw[ultra thick] (d)--(f);


\draw (e)--(L6);
\draw (e)--(L7);

\draw (f)--(L3);
\draw (f)--(L4);
\end{tikzpicture}
\hspace{3em}
\begin{tikzpicture}[  scale=0.6,  transform shape,
    every node/.style={
        circle,
        draw,
        minimum size=6mm,
        inner sep=0pt
    },
    thick
]

\node (a) at (-2, 1) {};
\node (b) at (-2,-1) {};
\node (d) at (-3,0) {};
\node (e) at (-4,1) {};
\node (f) at (-1,2) {};

\node (L1) at (-3,-2) {$1$};
\node (L2) at (-1,-2) {$2$};
\node (L3) at (0,3) {$3$};
\node (L4) at (0,1) {$4$};
\node (L5) at (-3,2) {$5$};
\node (L6) at (-5,0) {$6$};
\node (L7) at (-5,2) {$7$};

\draw (L1)--(b);
\draw (a)--(L5);

\draw (b)--(L2);
\draw (b)--(d);

\draw (d)--(e);
\draw (d)--(a);
\draw (a)--(f);


\draw (e)--(L6);
\draw (e)--(L7);

\draw (f)--(L3);
\draw (f)--(L4);
\end{tikzpicture}

\caption{
On the left a cladogram $X$. 
In the middle the cladogram $X'$ after the first half-step with the additional component in gray. We also already represent the randomly chosen edge where the leaf $\ell$ will be re-attached in the second half-step by making it thicker.
On the right the cladogram $X''$ obtained after the second half-step, which adds the (extended) leaf to the thick edge.}

\label{fig: illustration of one step of the cladogram chain}
\end{figure}
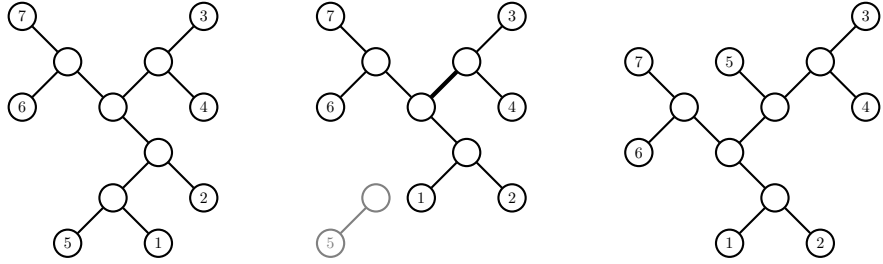

This defines a Markov chain $(X_t)_{t\geq 0} = (X^{(n)}_t)_{t\geq 0}$ on $\cC_n$ which is symmetric, irreducible, and aperiodic. Hence its unique stationary measure is the uniform
measure on cladograms, which we denote by $\pi_n$. Given $x,y \in \cC_n$ and an integer $t\geq 0$, we denote the $t$-step transition probabilities of the chain by $p_t(x,y) = \bbP(X_t = y \mid X_0 = x)$. In particular, $p_t(x, \cdot)$ is the distribution of the chain
starting at $x$ after $t$ steps.

\subsection{Main results}
 Before stating our result, we need to define the standard notions of distance to stationarity and mixing time, which we write directly  for our Markov chain $(X^{(n)}_t)_{t\geq 0}$ on the set $\cC_n$ of cladograms.

The worst-case total variation distance to  stationarity after $t$ steps is defined by
\begin{equation}\label{eq: def d n t}
   \textup{d}^{(n)}(t) := \max_{x\in \cC_n} \dtv(p_t(x, \cdot), \pi_n),
\end{equation}
where the total variation distance between two probability measures $\mu$ and $\nu$ on a finite set $S$ is given by $\dtv(\mu, \nu) = \max_{A\subseteq S}|\mu(A) - \nu(A)|$. The (total variation) mixing time 
is then defined by
\begin{equation}
    t_{\textup{mix}}^{(n)} = \min \ag t\geq 0 \mid  \textup{d}^{(n)}(t) \leq 1/4 \ad.
\end{equation}
We recall that the notation $a_n=\Theta(b_n)$ means that there exist constants $c>0$ and $C>0$ such that $c b_n \le a_n \le C b_n$ for all $n$ large enough. Our main result is the following theorem.
\begin{theorem}\label{thm: main mixing time cladograms}
   As $n\to \infty$, we have 
    \begin{equation}
        t_{\textup{mix}}^{(n)} = \Theta(n^2).
    \end{equation}
\end{theorem}
 In our case,
the lower bound $t_{\textup{mix}}^{(n)} \ge c n^2$ for some constant $c>0$ was established by Aldous in~\cite{Aldous2000MixingCladograms}. Our contribution here is the matching upper bound, up to a multiplicative constant.

\medskip

The motivation for studying the mixing time of such a Markov chain comes from Bayesian statistics in evolutionary biology, where Markov chains on cladrograms (usually with additional strucures, e.g.~branch lengths) are used to sample
posterior distributions on phylogenetic trees,
see, e.g.,~\cite{Yang2014Bayesian_phylogenetic}.
Knowing the mixing time of these chains is useful to have some theoretical guarantees on the precision of such algorithms and the chain on cladograms presently studied is the simplest example of this category. See also \cite{DiaconisHolmes2002RandomWalksOnTreesAndMatchings} for another type of dynamics on cladograms and perfect matchings. It would be interesting to obtain the precise mixing time behaviour for other models such as the SPR (Subtree Prune and Regraft) dynamics, see \cite{EvansWinter2006SubtreePruneAndRegraftSPR, SpadeHerbeiKubatko2014RelaxationTimeRootedPhylogeneticTrees}.

\subsection{Proof strategy and discussion}
Our proof of Theorem~\ref{thm: main mixing time cladograms}
uses a lift of the Markov chain to some colored cladograms and a coupling between two copies of the 
lifted chain started from different initial conditions.
The idea is that the colors will record the parts that are similar in the two copies of the chain.
Couplings are a standard method in the mixing time literature; see, e.g., \cite[Chapter 5]{LivreLevinPeres2019MarkovChainsAndMixingTimesSecondEdition}.

Our coupling is an improvement of the one introduced by Aldous to find his $O(n^3)$ upper
bound. 
The main novelty lies in the analysis of the coupling. Aldous focused on the evolution of the size of a single arbitrary component. In \cite[Remark 4.2]{Aldous2000MixingCladograms}, he explains why this should not be optimal.
\begin{quote}
“Note that our section 3.3 argument is plainly inefficient in that (in the jargon of colors above) we track just one color at a time to see if it takes over or goes extinct; and in the latter case we repeat with another color. It would be better to track all colors simultaneously, but the usual analysis of the voter model (using duality with coalescing random walk, here Kingman's coalescent) seems hard to adapt to our setting.”
\end{quote}
Our improved version of the coupling tracks all colors simultaneously, and we realized that the analysis can be performed by following the evolution
of a single real-valued statistic: the sum of the squares of the color sizes. This statistic has drift at least 1 on average, and thus must reach a value of order $n^2$ after a time of order $n^2$. 
This allows us to prove that, after a time of order $n^2$,
we reach a configuration with only three colors,
in which case the two cladograms must be identical.
Then a standard coupling lemma implies that the mixing time is order at most $n^2$, proving Theorem~\ref{thm: main mixing time cladograms}.\medskip

We would like to mention here another potential approach
which we were not able to make work, and would potentially lead
 to find not only the order of magnitude of the mixing time as we do here, but also the limit of the distance to stationarity at time $cn^2$ for any fixed $c>0$ (often referred to as \textit{limiting profile} in the mixing time literature).

The idea is to consider a projection of Aldous' chain by erasing the leaf labels;
this gives a Markov chain on {\em unlabelled} cladograms.
Then the first and second half-steps of the construction
become independent (their only connection in Aldous chain is via the label of the removed leaf),
and this chain is what is sometimes called a down-up chain in the literature.
There has been a growing literature on such chains in recent years (starting from~\cite{fulman2009commutation,borodin2009diffusions,petrov2009diffusions-Kingman}),
and in particular, the eigenvalues and eigenvectors of such chains under general conditions have been described in \cite{fulman2009commutation,FerayRiveraLopez2025UpDownChainsAndScalingLimits}.
This applies to the unlabeled version of Aldous chain, and one should be able from there to prove that the mixing time of the {\em unlabeled} chain is of order $\Theta(n^2)$,
and to compute the limiting profile.

The difficult step that we could not make work is to go from the unlabeled chain to the labeled one.
Intuitively, each step of Aldous' walk mixes the labels
in a similar way as a random-to-random chain on permutations.
This chain is known to mix in $\Theta(n \log n)$ steps (see, e.g., \cite{DiaconisSaloff-Coste1993comparisongroups}),
i.e.~much quicker than the chain on unlabeled trees. Therefore, we expect that when the underlying tree structure is mixed, the labels have already been mixed for a long time. Hence, the labeled and unlabeled chains should have mixing times of the same order, and even the same limiting profile.
Unfortunately, the evolution of the labels 
is difficult to study and is highly dependent on the underlying tree structure, and we could not figure out how to make this argument rigorous.

\section{A coupling that follows colored components in parallel}\label{s: coupling}

We now define our coupling. 
We refer to \cite[Chapter 5]{LivreLevinPeres2019MarkovChainsAndMixingTimesSecondEdition} for more details on coupling methods for Markov chains.
Let $n\geq 4$. Let $x,y\in \cC_n$ be two cladograms of size $n$. In the coupling we color some parts of the cladograms. We will therefore enrich the structure and define the dynamics of the coupling on \textit{colored} cladograms.

\subsection{Colored cladograms}\label{s: signed partially colored cladograms} 

Let $n\geq 4$. A colored cladogram consists informally of an uncolored tree to which are attached monochromatic trees with labels on leaves, as illustrated in Figure~\ref{fig: cladogram example 20 colored}.

\definecolor{Aeighteencolor}{HTML}{FFD54F}
\definecolor{Asevencolor}{HTML}{7986CB}
\definecolor{Asinglecolor}{HTML}{F4A261}
\definecolor{Asecondcolor}{HTML}{81C784}
\definecolor{Bsinglecolor}{HTML}{B39DDB}
\definecolor{Bfirstcolor}{HTML}{F48FB1}
\definecolor{Bmiddlecolor}{HTML}{E57373}
\definecolor{Bsecondcolor}{HTML}{64B5F6}
\definecolor{Dcommoncolor}{HTML}{4DB6AC}

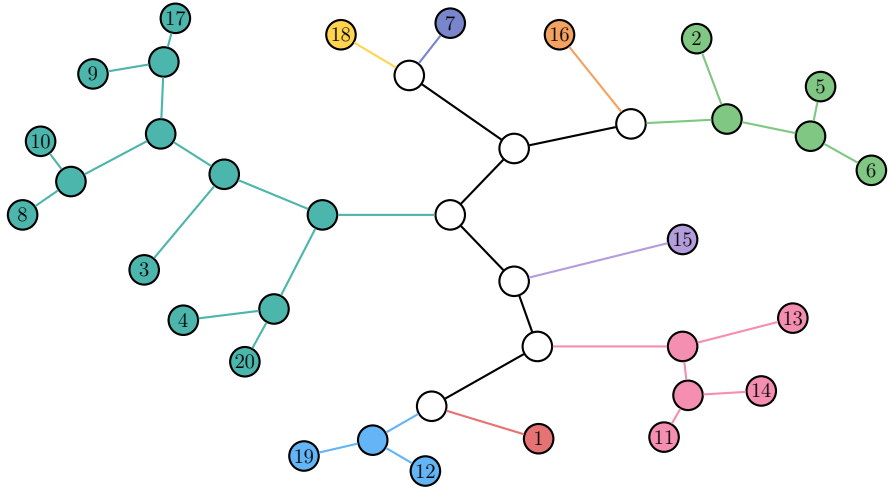
\begin{figure}[htbp]
\centering

\begin{tikzpicture}[
    scale=0.65,
    transform shape,
    every node/.style={
        circle,
        draw,
        minimum size=6mm,
        inner sep=0pt,
        font=\large
    },
    Aeighteen/.style={fill=Aeighteencolor},
    Aseven/.style={fill=Asevencolor},
    Asingle/.style={fill=Asinglecolor},
    Asecond/.style={fill=Asecondcolor},
    Bsingle/.style={fill=Bsinglecolor},
    Bfirst/.style={fill=Bfirstcolor},
    Bmiddle/.style={fill=Bmiddlecolor},
    Bsecond/.style={fill=Bsecondcolor},
    Dcommon/.style={fill=Dcommoncolor},
    thick
]

\def\radiusone{2.60}
\def\radiustwo{4.80}
\def\radiusthree{6.50}
\def\radiusfour{7.80}
\def\radiusfive{8.70}

\def\layoutrotation{-90}
\def\verticalcompression{0.6}

\newcommand{\radialposition}[2]{%
    ({#2*cos(#1+\layoutrotation)},%
     {\verticalcompression*#2*sin(#1+\layoutrotation)})%
}

\node (C) at (0,0) {};


\node (A0)
    at \radialposition{150}{\radiusone} {};

\node (A1)
    at \radialposition{190}{\radiustwo} {};
\node (A2)
    at \radialposition{130}{\radiustwo} {};

\node[Aeighteen] (L1)
    at \radialposition{200}{\radiusthree} {$18$};
\node[Aseven] (L2)
    at \radialposition{180}{\radiusthree} {$7$};
\node[Asingle] (L3)
    at \radialposition{160}{\radiusthree} {$16$};
\node[Asecond] (A3)
    at \radialposition{120}{\radiusthree} {};

\node[Asecond] (L4)
    at \radialposition{140}{\radiusfour} {$2$};
\node[Asecond] (A4)
    at \radialposition{110}{\radiusfour} {};

\node[Asecond] (L5)
    at \radialposition{120}{\radiusfive} {$5$};
\node[Asecond] (L6)
    at \radialposition{100}{\radiusfive} {$6$};


\node (B0)
    at \radialposition{30}{\radiusone} {};

\node[Bsingle] (L7)
    at \radialposition{80}{\radiustwo} {$15$};
\node (B1)
    at \radialposition{21.667}{\radiustwo} {};

\node[Bfirst] (B2)
    at \radialposition{46.667}{\radiusthree} {};
\node (B4)
    at \radialposition{-3.333}{\radiusthree} {};

\node[Bfirst] (L8)
    at \radialposition{63.333}{\radiusfour} {$13$};
\node[Bfirst] (B3)
    at \radialposition{38.333}{\radiusfour} {};
\node[Bmiddle] (L11)
    at \radialposition{13.333}{\radiusfour} {$1$};
\node[Bsecond] (B5)
    at \radialposition{-11.667}{\radiusfour} {};

\node[Bfirst] (L9)
    at \radialposition{46.667}{\radiusfive} {$14$};
\node[Bfirst] (L10)
    at \radialposition{30}{\radiusfive} {$11$};
\node[Bsecond] (L12)
    at \radialposition{-3.333}{\radiusfive} {$12$};
\node[Bsecond] (L13)
    at \radialposition{-20}{\radiusfive} {$19$};


\node[Dcommon] (D0)
    at \radialposition{270}{\radiusone} {};

\node[Dcommon] (D1)
    at \radialposition{311.667}{\radiustwo} {};
\node[Dcommon] (D2)
    at \radialposition{253.333}{\radiustwo} {};

\node[Dcommon] (L14)
    at \radialposition{320}{\radiusthree} {$20$};
\node[Dcommon] (L15)
    at \radialposition{303.333}{\radiusthree} {$4$};
\node[Dcommon] (L16)
    at \radialposition{286.667}{\radiusthree} {$3$};
\node[Dcommon] (D3)
    at \radialposition{245}{\radiusthree} {};

\node[Dcommon] (D4)
    at \radialposition{261.667}{\radiusfour} {};
\node[Dcommon] (D5)
    at \radialposition{228.333}{\radiusfour} {};

\node[Dcommon] (L17)
    at \radialposition{270}{\radiusfive} {$8$};
\node[Dcommon] (L18)
    at \radialposition{253.333}{\radiusfive} {$10$};
\node[Dcommon] (L19)
    at \radialposition{236.667}{\radiusfive} {$9$};
\node[Dcommon] (L20)
    at \radialposition{220}{\radiusfive} {$17$};


\foreach \u/\v in {
    C/A0,A0/A1,A0/A2,
    C/B0,B0/B1,B1/B4}
{
    \draw (\u)--(\v);
}

\draw[draw=Aeighteencolor] (A1)--(L1);

\draw[draw=Asevencolor] (A1)--(L2);

\draw[draw=Asinglecolor] (A2)--(L3);

\foreach \u/\v in {
    A2/A3,
    A3/L4,A3/A4,
    A4/L5,A4/L6}
{
    \draw[draw=Asecondcolor] (\u)--(\v);
}

\draw[draw=Bsinglecolor] (B0)--(L7);

\foreach \u/\v in {
    B1/B2,
    B2/L8,B2/B3,
    B3/L9,B3/L10}
{
    \draw[draw=Bfirstcolor] (\u)--(\v);
}

\draw[draw=Bmiddlecolor] (B4)--(L11);

\foreach \u/\v in {
    B4/B5,
    B5/L12,B5/L13}
{
    \draw[draw=Bsecondcolor] (\u)--(\v);
}

\foreach \u/\v in {
    C/D0,
    D0/D1,D0/D2,
    D1/L14,D1/L15,
    D2/L16,D2/D3,
    D3/D4,D3/D5,
    D4/L17,D4/L18,
    D5/L19,D5/L20}
{
    \draw[draw=Dcommoncolor] (\u)--(\v);
}

\end{tikzpicture}

\caption{A colored  cladogram with 20 leaves. The groups of leaves having the same color are  $\ag 16 \ad$, $\ag 7 \ad$, $\ag 18 \ad$,  $\ag 17, 9, 10, 8, 3, 4, 20 \ad$, $\ag 19, 12 \ad$, $\ag 1 \ad$, $\ag 11, 14, 13 \ad$, $\ag 15 \ad$, $\ag 6, 5, 2 \ad$, and the vertices up to the common ancestor of these leaves take the same color, as well as all adjacent edges to these vertices.}
\label{fig: cladogram example 20 colored}

\end{figure}

More formally, a colored cladogram of size $n$ is a cladogram where each vertex and edge is given an additional label (on top of those from the leaves) in $\ag 0, \ldots, n\ad$, where we think of the label 0 as “uncolored” and the labels from $1$ to $n$ as colors. Moreover, we require that
\begin{itemize}
    \item all leaves of the cladogram have a color in $\ag 1, \ldots, n\ad$;
    \item for all $i \ge 0$ (note that $i=0$ is included), vertices and edges of color $i$ form a (possibly empty) connected component of the cladogram, that we denote by $\cP_i$;
    \item each nonempty colored component $\cP_i$ (for $i\geq 1$) contains a unique edge $\ag u, v\ad$ of color $i$ such that $u$ has color $i$ and $v$ has color 0,
and all other edges of $\cP_i$ have both extremities of color $i$.
\end{itemize}

We denote the set of all colored cladograms of size $n$ by $\widetilde{\cC}_n$.
Let $\AddColors \du \cC_n \to \widetilde{\cC}_n$ be the application which transforms a cladogram into the colored cladogram where each leaf $i$ and its incident edge are given the color $i$, and the other vertices and edges are given the color $0$.

\subsection{Colored steps}\label{s: colored steps}

As explained in Section~\ref{s: the model},  a step of Aldous' walk on cladograms consists in removing a leaf uniformly at random and reinserting it on a uniform random edge. 
We now lift this operation to define a Markov chain on colored cladograms. Again, a step of the Markov chain is divided into two half-steps.
This construction is illustrated in Figures~\ref{fig: a colored step} and~\ref{fig: another colored step}.
Let $X\in \widetilde{\cC}_n$ be a colored cladogram with at least 4 colors.

\definecolor{onecolor}{HTML}{F4A261}
\definecolor{twocolor}{HTML}{FFD54F}
\definecolor{fivecolor}{HTML}{64B5F6}
\definecolor{groupcolor}{HTML}{81C784}

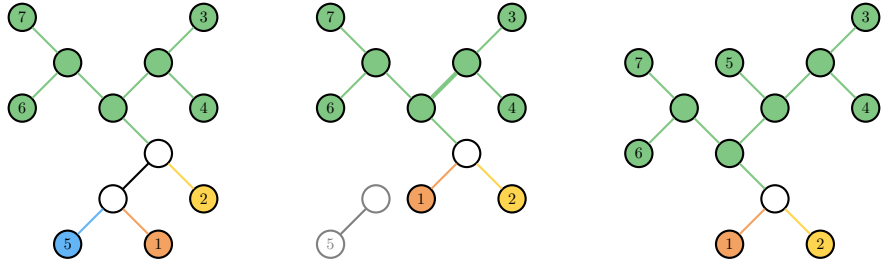
\begin{figure}[ht]
\centering

\begin{tikzpicture}[
    scale=0.6,
    transform shape,
    every node/.style={
        circle,
        draw,
        minimum size=6mm,
        inner sep=0pt
    },
    one/.style={fill=onecolor},
    two/.style={fill=twocolor},
    five/.style={fill=fivecolor},
    group/.style={fill=groupcolor},
    thick
]

\node (a) at (-2,-1) {};
\node (b) at (-1,0) {};
\node[group] (d) at (-2,1) {};
\node[group] (e) at (-3,2) {};
\node[group] (f) at (-1,2) {};

\node[one] (L1) at (-1,-2) {$1$};
\node[two] (L2) at (0,-1) {$2$};
\node[group] (L3) at (0,3) {$3$};
\node[group] (L4) at (0,1) {$4$};
\node[five] (L5) at (-3,-2) {$5$};
\node[group] (L6) at (-4,1) {$6$};
\node[group] (L7) at (-4,3) {$7$};

\draw[draw=onecolor] (L1)--(a);
\draw[draw=fivecolor] (a)--(L5);
\draw[draw=twocolor] (b)--(L2);

\draw (a)--(b);

\foreach \u/\v in {
    b/d,
    d/e,d/f,
    e/L6,e/L7,
    f/L3,f/L4}
{
    \draw[draw=groupcolor] (\u)--(\v);
}

\end{tikzpicture}
\hspace{3em}
\begin{tikzpicture}[
    scale=0.6,
    transform shape,
    every node/.style={
        circle,
        draw,
        minimum size=6mm,
        inner sep=0pt
    },
    one/.style={fill=onecolor},
    two/.style={fill=twocolor},
    group/.style={fill=groupcolor},
    removed/.style={
        draw=gray,
        text=gray
    },
    thick
]

\node[removed] (a) at (-3,-1) {};
\node (b) at (-1,0) {};
\node[group] (d) at (-2,1) {};
\node[group] (e) at (-3,2) {};
\node[group] (f) at (-1,2) {};

\node[one] (L1) at (-2,-1) {$1$};
\node[two] (L2) at (0,-1) {$2$};
\node[group] (L3) at (0,3) {$3$};
\node[group] (L4) at (0,1) {$4$};
\node[removed] (L5) at (-4,-2) {$5$};
\node[group] (L6) at (-4,1) {$6$};
\node[group] (L7) at (-4,3) {$7$};

\draw[draw=onecolor] (L1)--(b);

\draw[gray] (a)--(L5);

\draw[draw=twocolor] (b)--(L2);

\draw[draw=groupcolor] (b)--(d);
\draw[draw=groupcolor] (d)--(e);

\draw[draw=groupcolor,ultra thick] (d)--(f);

\foreach \u/\v in {
    e/L6,e/L7,
    f/L3,f/L4}
{
    \draw[draw=groupcolor] (\u)--(\v);
}

\end{tikzpicture}
\hspace{3em}
\begin{tikzpicture}[
    scale=0.6,
    transform shape,
    every node/.style={
        circle,
        draw,
        minimum size=6mm,
        inner sep=0pt
    },
    one/.style={fill=onecolor},
    two/.style={fill=twocolor},
    group/.style={fill=groupcolor},
    thick
]

\node[group] (a) at (-2,1) {};
\node (b) at (-2,-1) {};
\node[group] (d) at (-3,0) {};
\node[group] (e) at (-4,1) {};
\node[group] (f) at (-1,2) {};

\node[one] (L1) at (-3,-2) {$1$};
\node[two] (L2) at (-1,-2) {$2$};
\node[group] (L3) at (0,3) {$3$};
\node[group] (L4) at (0,1) {$4$};

\node[group] (L5) at (-3,2) {$5$};

\node[group] (L6) at (-5,0) {$6$};
\node[group] (L7) at (-5,2) {$7$};

\draw[draw=onecolor] (L1)--(b);
\draw[draw=twocolor] (b)--(L2);

\draw[draw=groupcolor] (b)--(d);
\draw[draw=groupcolor] (d)--(e);

\draw[draw=groupcolor] (d)--(a);
\draw[draw=groupcolor] (a)--(L5);
\draw[draw=groupcolor] (a)--(f);

\foreach \u/\v in {
    e/L6,e/L7,
    f/L3,f/L4}
{
    \draw[draw=groupcolor] (\u)--(\v);
}

\end{tikzpicture}

\caption{
Illustration of a colored step.
}
\label{fig: a colored step}

\end{figure}

\definecolor{grouponecolor}{HTML}{F4A261}
\definecolor{groupsixcolor}{HTML}{81C784}
\definecolor{threecolor}{HTML}{64B5F6}
\definecolor{fourcolor}{HTML}{B39DDB}
\definecolor{newfivecolor}{HTML}{E57373}

\begin{figure}[ht]
\centering


\begin{tikzpicture}[
    scale=0.6,
    transform shape,
    every node/.style={
        circle,
        draw,
        minimum size=6mm,
        inner sep=0pt
    },
    groupone/.style={fill=grouponecolor},
    groupsix/.style={fill=groupsixcolor},
    three/.style={fill=threecolor},
    four/.style={fill=fourcolor},
    thick
]

\node[groupone] (a) at (-2,-1) {};
\node[groupone] (b) at (-1,0) {};
\node (d) at (-2,1) {};
\node[groupsix] (e) at (-3,2) {};
\node (f) at (-1,2) {};

\node[groupone] (L1) at (-1,-2) {$1$};
\node[groupone] (L2) at (0,-1) {$2$};
\node[three] (L3) at (0,3) {$3$};
\node[four] (L4) at (0,1) {$4$};
\node[groupone] (L5) at (-3,-2) {$5$};
\node[groupsix] (L6) at (-4,1) {$6$};
\node[groupsix] (L7) at (-4,3) {$7$};

\foreach \u/\v in {
    L1/a,a/L5,a/b,b/L2,b/d}
{
    \draw[draw=grouponecolor] (\u)--(\v);
}

\foreach \u/\v in {
    d/e,e/L6,e/L7}
{
    \draw[draw=groupsixcolor] (\u)--(\v);
}

\draw[draw=threecolor] (f)--(L3);
\draw[draw=fourcolor] (f)--(L4);

\draw (d)--(f);

\end{tikzpicture}%
\hspace{3em}%
\begin{tikzpicture}[
    scale=0.6,
    transform shape,
    every node/.style={
        circle,
        draw,
        minimum size=6mm,
        inner sep=0pt
    },
    groupone/.style={fill=grouponecolor},
    groupsix/.style={fill=groupsixcolor},
    three/.style={fill=threecolor},
    four/.style={fill=fourcolor},
    removed/.style={
        draw=gray,
        text=gray
    },
    thick
]

\node[removed] (a) at (-3,-1) {};
\node[groupone] (b) at (-1,0) {};
\node (d) at (-2,1) {};
\node[groupsix] (e) at (-3,2) {};
\node (f) at (-1,2) {};

\node[groupone] (L1) at (-2,-1) {$1$};
\node[groupone] (L2) at (0,-1) {$2$};
\node[three] (L3) at (0,3) {$3$};
\node[four] (L4) at (0,1) {$4$};
\node[removed] (L5) at (-4,-2) {$5$};
\node[groupsix] (L6) at (-4,1) {$6$};
\node[groupsix] (L7) at (-4,3) {$7$};

\draw[gray] (a)--(L5);

\foreach \u/\v in {
    L1/b,b/L2,b/d}
{
    \draw[draw=grouponecolor] (\u)--(\v);
}

\foreach \u/\v in {
    d/e,e/L6,e/L7}
{
    \draw[draw=groupsixcolor] (\u)--(\v);
}

\draw[draw=threecolor] (f)--(L3);
\draw[draw=fourcolor] (f)--(L4);

\draw[ultra thick] (d)--(f);

\end{tikzpicture}%
\hspace{3em}%
\begin{tikzpicture}[
    scale=0.6,
    transform shape,
    every node/.style={
        circle,
        draw,
        minimum size=6mm,
        inner sep=0pt
    },
    groupone/.style={fill=grouponecolor},
    groupsix/.style={fill=groupsixcolor},
    three/.style={fill=threecolor},
    four/.style={fill=fourcolor},
    newfive/.style={fill=newfivecolor},
    thick
]

\node (a) at (-2,1) {};
\node[groupone] (b) at (-2,-1) {};
\node (d) at (-3,0) {};
\node[groupsix] (e) at (-4,1) {};
\node (f) at (-1,2) {};

\node[groupone] (L1) at (-3,-2) {$1$};
\node[groupone] (L2) at (-1,-2) {$2$};
\node[three] (L3) at (0,3) {$3$};
\node[four] (L4) at (0,1) {$4$};

\node[newfive] (L5) at (-3,2) {$5$};

\node[groupsix] (L6) at (-5,0) {$6$};
\node[groupsix] (L7) at (-5,2) {$7$};

\foreach \u/\v in {
    L1/b,b/L2,b/d}
{
    \draw[draw=grouponecolor] (\u)--(\v);
}

\foreach \u/\v in {
    d/e,e/L6,e/L7}
{
    \draw[draw=groupsixcolor] (\u)--(\v);
}

\draw[draw=threecolor] (f)--(L3);
\draw[draw=fourcolor] (f)--(L4);

\draw[draw=newfivecolor] (a)--(L5);

\draw (d)--(a);
\draw (a)--(f);

\end{tikzpicture}

\caption{
Illustration of another colored step.
}
\label{fig: another colored step}

\end{figure}
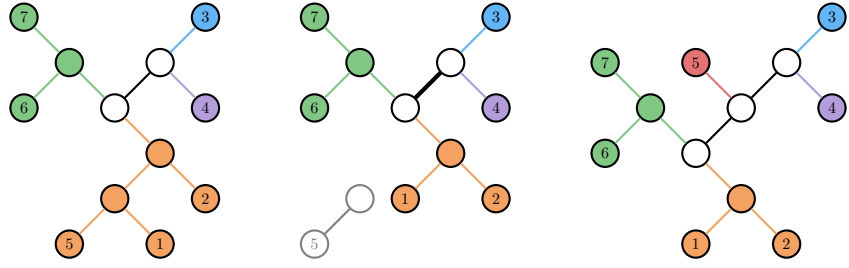

\paragraph*{First half-step.} We first pick a uniformly random leaf $\ell$ in $X$. Denote the neighbour of $\ell$ by $u$, and the neighbours of $u$ by $u_1$ and $u_2$. 
Then we replace the two edges $\ag u,u_1 \ad$ and $\ag u, u_2 \ad$ by a single edge $\ag u_1, u_2\ad$, which takes the color of $\ag u,u_1 \ad$ or $\ag u, u_2 \ad$ if at least one of them is colored, and remains uncolored if both of them are uncolored.

The assumption that there are at least four colored components is important for this step. It implies that the connected uncolored component has at least two vertices, and in particular that each uncolored vertex is connected to at least one uncolored neighbour, via an uncolored edge.
It is therefore not possible that the edges $\ag u,u_1 \ad$ and $\ag u, u_2 \ad$ have different colors $i,j \ge 1$: either only one is colored, as in Figure~\ref{fig: a colored step}, or the two have the same color, as in Figure~\ref{fig: another colored step}.
Hence, $\{u_1,u_2\}$ is either uncolored, or has a well-defined color.

This first half-step yields a colored cladogram $X'$ with $n-1$ leaves, and an additional part consisting of $u, \ell$ and $\ag u, \ell \ad$.

\paragraph*{Second half-step.} We pick an edge $\ag v_1, v_2 \ad$ of $X'$ uniformly at random and, calling $i$ the color of the edge $\ag v_1, v_2\ad$, replace this edge by two edges $\ag v_1, u\ad$ and $\ag v_2, u\ad$, which both also have the color $i$.
We also give the following colors to the reinserted leaf:
\begin{itemize}
    \item if $i\geq 1$, then the vertex $u$, the leaf  $\ell$, and the edge $\ag u, \ell\ad$ all take  color $i$ (as in Figure~\ref{fig: a colored step});
    \item if $i=0$, then the vertex $u$ takes the color 0, but the leaf $\ell$ and the edge $\ag u, \ell \ad$ take an arbitrary unused color in $X'$ (as in Figure~\ref{fig: another colored step}), e.g.~$\min\{j:\, \cP_j(X') = \emptyset\}$ to define the step uniquely.
    \end{itemize}
This defines a colored cladogram $X''$ in $\widetilde{\cC}_n$ with at least $3$ colors.

\subsection{Definition of the multicolor coupling}\label{s: definition of the two-step coupling}
Let $n\geq 4$, let $x,y\in \cC_n$, and set $X_0 = \AddColors(x)$ and $Y_0 = \AddColors(y)$.
In words, we start with colored cladograms where only leaves and their incident edges are colored,
with $n$ different colors.
We now define a coupled dynamics $(X_t)_{t\geq 0}$ and $(Y_t)_{t\geq 0}$,
such that both $X_t$ and $Y_t$ individually are Markov chains with transitions as described in Section~\ref{s: colored steps}
and initial conditions $X_0$ and $Y_0$, respectively.

\medskip

The coupling which we will construct has the following invariant. For every color $i\in [n] = \ag 1, \ldots, n\ad$ (we emphasize that $i=0$ is excluded here) and every $t\geq 0$, the colored components $\cP_i(X_t)$ and $\cP_i(Y_t)$ are identical, including the labels on their leaves. This holds by construction for $t=0$ (for which both $\cP_i(X_t)$ and $\cP_i(Y_t)$ consist of the leaf labeled $i$ and its incident edge), and it will be easy to see that this property is preserved along our construction.

\medskip

\definecolor{leafonecolor}{HTML}{F4A261}
\definecolor{leaftwocolor}{HTML}{FFD54F}
\definecolor{leafthreecolor}{HTML}{81C784}
\definecolor{leaffourcolor}{HTML}{4DB6AC}
\definecolor{leaffivecolor}{HTML}{64B5F6}
\definecolor{leafsixcolor}{HTML}{B39DDB}

\tikzset{
    sixcladogram/.style={
        scale=0.42,
        transform shape,
        every node/.style={
            circle,
            draw,
            minimum size=7mm,
            inner sep=0pt,
            font=\large
        },
        leafone/.style={fill=leafonecolor},
        leaftwo/.style={fill=leaftwocolor},
        leafthree/.style={fill=leafthreecolor},
        leaffour/.style={fill=leaffourcolor},
        leaffive/.style={fill=leaffivecolor},
        leafsix/.style={fill=leafsixcolor},
        thick
    }
}

\def\edgelen{1.8}

\begin{figure}[htbp]
\centering

\begin{tabular}{@{}c@{\hspace{6em}}c@{}}


\(X_0=\;\) \begin{tikzpicture}[     sixcladogram,     baseline=(current bounding box.center) ]

\node (v1) at (0,0) {};
\node (v2) at ([shift={(10:\edgelen)}]v1) {};
\node (v3) at ([shift={(340:\edgelen)}]v2) {};
\node (v4) at ([shift={(20:\edgelen)}]v3) {};

\node[leafone] (L1)
    at ([shift={(140:\edgelen)}]v1) {$1$};

\node[leaftwo] (L2)
    at ([shift={(235:\edgelen)}]v1) {$2$};

\node[leafthree] (L3)
    at ([shift={(105:\edgelen)}]v2) {$3$};

\node[leaffour] (L4)
    at ([shift={(280:\edgelen)}]v3) {$4$};

\node[leaffive] (L5)
    at ([shift={(60:\edgelen)}]v4) {$5$};

\node[leafsix] (L6)
    at ([shift={(320:\edgelen)}]v4) {$6$};

\draw (v1)--(v2);
\draw (v2)--(v3);
\draw (v3)--(v4);

\draw[draw=leafonecolor]   (v1)--(L1);
\draw[draw=leaftwocolor]   (v1)--(L2);
\draw[draw=leafthreecolor] (v2)--(L3);
\draw[draw=leaffourcolor]  (v3)--(L4);
\draw[draw=leaffivecolor]  (v4)--(L5);
\draw[draw=leafsixcolor]   (v4)--(L6);

\end{tikzpicture}
&

\(Y_0=\;\) \begin{tikzpicture}[     sixcladogram,     baseline=(current bounding box.center) ]

\node (c) at (0,0) {};

\node (v1) at ([shift={(170:\edgelen)}]c) {};
\node (v2) at ([shift={(50:\edgelen)}]c) {};
\node (v3) at ([shift={(290:\edgelen)}]c) {};

\node[leafone] (L1)
    at ([shift={(110:\edgelen)}]v1) {$1$};

\node[leaffour] (L4)
    at ([shift={(230:\edgelen)}]v1) {$4$};

\node[leafsix] (L6)
    at ([shift={(140:\edgelen)}]v2) {$6$};

\node[leaftwo] (L2)
    at ([shift={(350:\edgelen)}]v2) {$2$};

\node[leaffive] (L5)
    at ([shift={(350:\edgelen)}]v3) {$5$};

\node[leafthree] (L3)
    at ([shift={(200:\edgelen)}]v3) {$3$};

\draw (c)--(v1);
\draw (c)--(v2);
\draw (c)--(v3);

\draw[draw=leafonecolor]   (v1)--(L1);
\draw[draw=leaffourcolor]  (v1)--(L4);
\draw[draw=leafsixcolor]   (v2)--(L6);
\draw[draw=leaftwocolor]   (v2)--(L2);
\draw[draw=leaffivecolor]  (v3)--(L5);
\draw[draw=leafthreecolor] (v3)--(L3);

\end{tikzpicture}
\\[15mm]


\(X_1=\;\) \begin{tikzpicture}[     sixcladogram,     baseline=(current bounding box.center) ]

\node (v2) at (0,0) {};
\node (v3) at ([shift={(0:\edgelen)}]v2) {};
\node (v4) at ([shift={(20:\edgelen)}]v3) {};

\node[leaffour] (v1)
    at ([shift={(285:\edgelen)}]v3) {};

\node[leaftwo] (L2)
    at ([shift={(220:\edgelen)}]v2) {$2$};

\node[leafthree] (L3)
    at ([shift={(110:\edgelen)}]v2) {$3$};

\node[leaffive] (L5)
    at ([shift={(70:\edgelen)}]v4) {$5$};

\node[leafsix] (L6)
    at ([shift={(330:\edgelen)}]v4) {$6$};

\node[leaffour] (L4)
    at ([shift={(325:\edgelen)}]v1) {$4$};

\node[leaffour] (L1)
    at ([shift={(210:\edgelen)}]v1) {$1$};

\draw (v2)--(v3);
\draw (v3)--(v4);

\draw[draw=leaftwocolor]   (v2)--(L2);
\draw[draw=leafthreecolor] (v2)--(L3);
\draw[draw=leaffivecolor]  (v4)--(L5);
\draw[draw=leafsixcolor]   (v4)--(L6);
\draw[draw=leaffourcolor]  (v3)--(v1);
\draw[draw=leaffourcolor]  (v1)--(L4);
\draw[draw=leaffourcolor]  (v1)--(L1);

\end{tikzpicture}
&

\(Y_1=\;\) \begin{tikzpicture}[     sixcladogram,     baseline=(current bounding box.center) ]

\node (c) at (0,0) {};

\node[leaffour] (v1)
    at ([shift={(170:\edgelen)}]c) {};

\node (v2)
    at ([shift={(50:\edgelen)}]c) {};

\node (v3)
    at ([shift={(290:\edgelen)}]c) {};

\node[leaffour] (L1)
    at ([shift={(110:\edgelen)}]v1) {$1$};

\node[leaffour] (L4)
    at ([shift={(230:\edgelen)}]v1) {$4$};

\node[leafsix] (L6)
    at ([shift={(140:\edgelen)}]v2) {$6$};

\node[leaftwo] (L2)
    at ([shift={(350:\edgelen)}]v2) {$2$};

\node[leaffive] (L5)
    at ([shift={(350:\edgelen)}]v3) {$5$};

\node[leafthree] (L3)
    at ([shift={(200:\edgelen)}]v3) {$3$};

\draw[draw=leaffourcolor] (c)--(v1);
\draw (c)--(v2);
\draw (c)--(v3);

\draw[draw=leaffourcolor]  (v1)--(L1);
\draw[draw=leaffourcolor]  (v1)--(L4);
\draw[draw=leafsixcolor]   (v2)--(L6);
\draw[draw=leaftwocolor]   (v2)--(L2);
\draw[draw=leaffivecolor]  (v3)--(L5);
\draw[draw=leafthreecolor] (v3)--(L3);

\end{tikzpicture}
\\[15mm]


\(X_2=\;\) \begin{tikzpicture}[     sixcladogram,     baseline=(current bounding box.center) ]

\node (v3) at (0,0) {};
\node (v4) at ([shift={(20:\edgelen)}]v3) {};

\node[leaffour] (v1)
    at ([shift={(180:\edgelen)}]v3) {};

\node[leaffour] (v2)
    at ([shift={(200:\edgelen)}]v1) {};

\node[leaffour] (L4)
    at ([shift={(200:\edgelen)}]v2) {$4$};

\node[leaffour] (L1)
    at ([shift={(140:\edgelen)}]v1) {$1$};

\node[leaffour] (L2)
    at ([shift={(140:\edgelen)}]v2) {$2$};

\node[leafthree] (L3)
    at ([shift={(110:\edgelen)}]v3) {$3$};

\node[leaffive] (L5)
    at ([shift={(70:\edgelen)}]v4) {$5$};

\node[leafsix] (L6)
    at ([shift={(330:\edgelen)}]v4) {$6$};

\draw (v3)--(v4);

\draw[draw=leafthreecolor] (v3)--(L3);
\draw[draw=leaffivecolor]  (v4)--(L5);
\draw[draw=leafsixcolor]   (v4)--(L6);

\draw[draw=leaffourcolor] (v3)--(v1);
\draw[draw=leaffourcolor] (v1)--(L1);
\draw[draw=leaffourcolor] (v1)--(v2);
\draw[draw=leaffourcolor] (v2)--(L2);
\draw[draw=leaffourcolor] (v2)--(L4);

\end{tikzpicture}
&

\(Y_2=\;\) \begin{tikzpicture}[     sixcladogram,     baseline=(current bounding box.center) ]

\node (c) at (0,0) {};

\node[leaffour] (v1)
    at ([shift={(180:\edgelen)}]c) {};

\node[leaffour] (v2)
    at ([shift={(200:\edgelen)}]v1) {};

\node[leaffour] (L1)
    at ([shift={(90:\edgelen)}]v1) {$1$};

\node[leaffour] (L4)
    at ([shift={(120:\edgelen)}]v2) {$4$};

\node[leaffour] (L2)
    at ([shift={(200:\edgelen)}]v2) {$2$};

\node (v3)
    at ([shift={(340:\edgelen)}]c) {};

\node[leaffive] (L5)
    at ([shift={(40:\edgelen)}]v3) {$5$};

\node[leafthree] (L3)
    at ([shift={(310:\edgelen)}]v3) {$3$};

\node[leafsix] (L6)
    at ([shift={(70:\edgelen)}]c) {$6$};

\draw[draw=leaffourcolor] (c)--(v1);
\draw[draw=leaffourcolor] (v1)--(L1);
\draw[draw=leaffourcolor] (v1)--(v2);
\draw[draw=leaffourcolor] (v2)--(L2);
\draw[draw=leaffourcolor] (v2)--(L4);

\draw (c)--(v3);
\draw[draw=leaffivecolor]  (v3)--(L5);
\draw[draw=leafthreecolor] (v3)--(L3);
\draw[draw=leafsixcolor]   (c)--(L6);

\end{tikzpicture}
\\[15mm]


\(X_3=\;\) \begin{tikzpicture}[     sixcladogram,     baseline=(current bounding box.center) ]

\node (v3) at (0,0) {};
\node (v4) at ([shift={(350:\edgelen)}]v3) {};

\node[leaffour] (v1)
    at ([shift={(200:\edgelen)}]v3) {};

\node[leaffour] (L1)
    at ([shift={(140:\edgelen)}]v1) {$1$};

\node[leaffour] (L4)
    at ([shift={(200:\edgelen)}]v1) {$4$};

\node[leafthree] (L3)
    at ([shift={(130:\edgelen)}]v3) {$3$};

\node[leaffive] (v2)
    at ([shift={(50:\edgelen)}]v4) {};

\node[leaffive] (L5)
    at ([shift={(50:\edgelen)}]v2) {$5$};

\node[leaffive] (L2)
    at ([shift={(320:\edgelen)}]v2) {$2$};

\node[leafsix] (L6)
    at ([shift={(310:\edgelen)}]v4) {$6$};

\draw (v3)--(v4);

\draw[draw=leafthreecolor] (v3)--(L3);
\draw[draw=leafsixcolor]   (v4)--(L6);

\draw[draw=leaffourcolor] (v3)--(v1);
\draw[draw=leaffourcolor] (v1)--(L1);
\draw[draw=leaffourcolor] (v1)--(L4);

\draw[draw=leaffivecolor] (v4)--(v2);
\draw[draw=leaffivecolor] (v2)--(L5);
\draw[draw=leaffivecolor] (v2)--(L2);

\end{tikzpicture}
&

\(Y_3=\;\) \begin{tikzpicture}[     sixcladogram,     baseline=(current bounding box.center) ]

\node (c) at (0,0) {};

\node[leaffour] (v1)
    at ([shift={(180:\edgelen)}]c) {};

\node[leaffour] (L1)
    at ([shift={(120:\edgelen)}]v1) {$1$};

\node[leaffour] (L4)
    at ([shift={(240:\edgelen)}]v1) {$4$};

\node (v3)
    at ([shift={(340:\edgelen)}]c) {};

\node[leafthree] (L3)
    at ([shift={(270:\edgelen)}]v3) {$3$};

\node[leaffive] (v2)
    at ([shift={(10:\edgelen)}]v3) {};

\node[leaffive] (L5)
    at ([shift={(10:\edgelen)}]v2) {$5$};

\node[leaffive] (L2)
    at ([shift={(300:\edgelen)}]v2) {$2$};

\node[leafsix] (L6)
    at ([shift={(65:\edgelen)}]c) {$6$};

\draw[draw=leaffourcolor] (c)--(v1);
\draw[draw=leaffourcolor] (v1)--(L1);
\draw[draw=leaffourcolor] (v1)--(L4);

\draw (c)--(v3);
\draw[draw=leafthreecolor] (v3)--(L3);

\draw[draw=leaffivecolor] (v3)--(v2);
\draw[draw=leaffivecolor] (v2)--(L5);
\draw[draw=leaffivecolor] (v2)--(L2);

\draw[draw=leafsixcolor] (c)--(L6);

\end{tikzpicture}
\\[15mm]


\(X_4=\;\) \begin{tikzpicture}[     sixcladogram,     baseline=(current bounding box.center) ]

\node (v3) at (0,0) {};

\node (v2)
    at ([shift={(0:\edgelen)}]v3) {};

\node (v4)
    at ([shift={(10:\edgelen)}]v2) {};

\node[leaffour] (v1)
    at ([shift={(200:\edgelen)}]v3) {};

\node[leaffour] (L1)
    at ([shift={(140:\edgelen)}]v1) {$1$};

\node[leaffour] (L4)
    at ([shift={(200:\edgelen)}]v1) {$4$};

\node[leafthree] (L3)
    at ([shift={(130:\edgelen)}]v3) {$3$};

\node[leaffive] (L2)
    at ([shift={(50:\edgelen)}]v4) {$2$};

\node[leafsix] (L6)
    at ([shift={(310:\edgelen)}]v4) {$6$};

\node[leafone] (L5)
    at ([shift={(100:\edgelen)}]v2) {$5$};

\draw (v3)--(v2);
\draw (v2)--(v4);

\draw[draw=leafthreecolor] (v3)--(L3);
\draw[draw=leafsixcolor]   (v4)--(L6);
\draw[draw=leaffivecolor]  (v4)--(L2);

\draw[draw=leaffourcolor] (v3)--(v1);
\draw[draw=leaffourcolor] (v1)--(L1);
\draw[draw=leaffourcolor] (v1)--(L4);

\draw[draw=leafonecolor] (v2)--(L5);

\end{tikzpicture}
&

\(Y_4=\;\) \begin{tikzpicture}[     sixcladogram,     baseline=(current bounding box.center) ]

\node (c) at (0,0) {};

\node (v2)
    at ([shift={(20:\edgelen)}]c) {};

\node (v3)
    at ([shift={(0:\edgelen)}]v2) {};

\node[leaffour] (v1)
    at ([shift={(180:\edgelen)}]c) {};

\node[leaffour] (L1)
    at ([shift={(130:\edgelen)}]v1) {$1$};

\node[leaffour] (L4)
    at ([shift={(220:\edgelen)}]v1) {$4$};

\node[leafsix] (L6)
    at ([shift={(90:\edgelen)}]c) {$6$};

\node[leaffive] (L2)
    at ([shift={(300:\edgelen)}]v3) {$2$};

\node[leafthree] (L3)
    at ([shift={(40:\edgelen)}]v3) {$3$};

\node[leafone] (L5)
    at ([shift={(270:\edgelen)}]v2) {$5$};

\draw (c)--(v2);
\draw (v2)--(v3);

\draw[draw=leaffourcolor] (c)--(v1);
\draw[draw=leaffourcolor] (v1)--(L1);
\draw[draw=leaffourcolor] (v1)--(L4);

\draw[draw=leafsixcolor]   (c)--(L6);
\draw[draw=leaffivecolor]  (v3)--(L2);
\draw[draw=leafthreecolor] (v3)--(L3);
\draw[draw=leafonecolor]   (v2)--(L5);

\end{tikzpicture}
\\[15mm]


\(X_5=\;\) \begin{tikzpicture}[     sixcladogram,     baseline=(current bounding box.center) ]

\node (v3) at (0,0) {};

\node (v2)
    at ([shift={(-10:\edgelen)}]v3) {};

\node[leaffour] (v1)
    at ([shift={(200:\edgelen)}]v3) {};

\node[leaffour] (L1)
    at ([shift={(140:\edgelen)}]v1) {$1$};

\node[leaffour] (L4)
    at ([shift={(200:\edgelen)}]v1) {$4$};
\node[leafthree] (L3)
    at ([shift={(130:\edgelen)}]v3) {$3$};

\node[leaffive] (L2)
    at ([shift={(310:\edgelen)}]v2) {$2$};

\node[leafone] (v4)
    at ([shift={(20:\edgelen)}]v2) {};

\node[leafone] (L5)
    at ([shift={(50:\edgelen)}]v4) {$5$};

\node[leafone] (L6)
    at ([shift={(-40:\edgelen)}]v4) {$6$};

\draw (v3)--(v2);

\draw[draw=leafthreecolor] (v3)--(L3);
\draw[draw=leaffivecolor]  (v2)--(L2);

\draw[draw=leaffourcolor] (v3)--(v1);
\draw[draw=leaffourcolor] (v1)--(L1);
\draw[draw=leaffourcolor] (v1)--(L4);

\draw[draw=leafonecolor] (v2)--(v4);
\draw[draw=leafonecolor] (v4)--(L5);
\draw[draw=leafonecolor] (v4)--(L6);

\end{tikzpicture}
&

\(Y_5=\;\) \begin{tikzpicture}[     sixcladogram,     baseline=(current bounding box.center) ]

\node (v2) at (0,0) {};

\node (v3)
    at ([shift={(20:\edgelen)}]v2) {};

\node[leaffour] (v1)
    at ([shift={(180:\edgelen)}]v2) {};

\node[leaffour] (L1)
    at ([shift={(130:\edgelen)}]v1) {$1$};

\node[leaffour] (L4)
    at ([shift={(220:\edgelen)}]v1) {$4$};

\node[leaffive] (L2)
    at ([shift={(300:\edgelen)}]v3) {$2$};

\node[leafthree] (L3)
    at ([shift={(40:\edgelen)}]v3) {$3$};

\node[leafone] (c)
    at ([shift={(270:\edgelen)}]v2) {};

\node[leafone] (L5)
    at ([shift={(200:\edgelen)}]c) {$5$};

\node[leafone] (L6)
    at ([shift={(340:\edgelen)}]c) {$6$};

\draw (v2)--(v3);

\draw[draw=leaffourcolor] (v2)--(v1);
\draw[draw=leaffourcolor] (v1)--(L1);
\draw[draw=leaffourcolor] (v1)--(L4);

\draw[draw=leaffivecolor]  (v3)--(L2);
\draw[draw=leafthreecolor] (v3)--(L3);

\draw[draw=leafonecolor] (v2)--(c);
\draw[draw=leafonecolor] (c)--(L5);
\draw[draw=leafonecolor] (c)--(L6);

\end{tikzpicture}
\\[15mm]


\(X_6=\;\) \begin{tikzpicture}[     sixcladogram,     baseline=(current bounding box.center) ]

\node (v2) at (0,0) {};

\node[leaffour] (v1)
    at ([shift={(200:\edgelen)}]v2) {};

\node[leaffour] (L1)
    at ([shift={(140:\edgelen)}]v1) {$1$};

\node[leaffour] (L4)
    at ([shift={(200:\edgelen)}]v1) {$4$};

\node[leaffive] (L2)
    at ([shift={(310:\edgelen)}]v2) {$2$};

\node[leafone] (v3)
    at ([shift={(50:\edgelen)}]v2) {};

\node[leafone] (v4)
    at ([shift={(20:\edgelen)}]v3) {};

\node[leafone] (L3)
    at ([shift={(140:\edgelen)}]v3) {$3$};

\node[leafone] (L5)
    at ([shift={(50:\edgelen)}]v4) {$5$};

\node[leafone] (L6)
    at ([shift={(-40:\edgelen)}]v4) {$6$};

\draw[draw=leaffourcolor] (v2)--(v1);
\draw[draw=leaffourcolor] (v1)--(L1);
\draw[draw=leaffourcolor] (v1)--(L4);

\draw[draw=leaffivecolor] (v2)--(L2);

\draw[draw=leafonecolor] (v2)--(v3);
\draw[draw=leafonecolor] (v3)--(v4);
\draw[draw=leafonecolor] (v3)--(L3);
\draw[draw=leafonecolor] (v4)--(L5);
\draw[draw=leafonecolor] (v4)--(L6);

\end{tikzpicture}
&

\(Y_6=\;\) \begin{tikzpicture}[     sixcladogram,     baseline=(current bounding box.center) ]

\node (v2) at (0,0) {};

\node[leaffour] (v1)
    at ([shift={(180:\edgelen)}]v2) {};

\node[leaffour] (L1)
    at ([shift={(130:\edgelen)}]v1) {$1$};

\node[leaffour] (L4)
    at ([shift={(220:\edgelen)}]v1) {$4$};

\node[leaffive] (L2)
    at ([shift={(80:\edgelen)}]v2) {$2$};

\node[leafone] (v3)
    at ([shift={(340:\edgelen)}]v2) {};

\node[leafone] (c)
    at ([shift={(20:\edgelen)}]v3) {};

\node[leafone] (L3)
    at ([shift={(220:\edgelen)}]v3) {$3$};

\node[leafone] (L5)
    at ([shift={(330:\edgelen)}]c) {$5$};

\node[leafone] (L6)
    at ([shift={(60:\edgelen)}]c) {$6$};

\draw[draw=leaffourcolor] (v2)--(v1);
\draw[draw=leaffourcolor] (v1)--(L1);
\draw[draw=leaffourcolor] (v1)--(L4);

\draw[draw=leaffivecolor] (v2)--(L2);

\draw[draw=leafonecolor] (v2)--(v3);
\draw[draw=leafonecolor] (v3)--(c);
\draw[draw=leafonecolor] (v3)--(L3);
\draw[draw=leafonecolor] (c)--(L5);
\draw[draw=leafonecolor] (c)--(L6);

\end{tikzpicture}

\end{tabular}

\caption{Some possible coupled colored steps in the coupling. In step $1$, the leaf $1$ is moved close to $4$ in both trees, and gets the same color as $1$. In step $4$, the leaf $5$ is moved in both trees to an uncolored edge and gets a new color.
In step $6$, the leaf $3$ is moved to an orange edge, and the number of colors goes down to $3$. Hence, $T=6$ here, and the trees $X_6$ and $Y_6$ are identical (as nonplane trees).}
\label{fig: coupling illustration}

\end{figure}
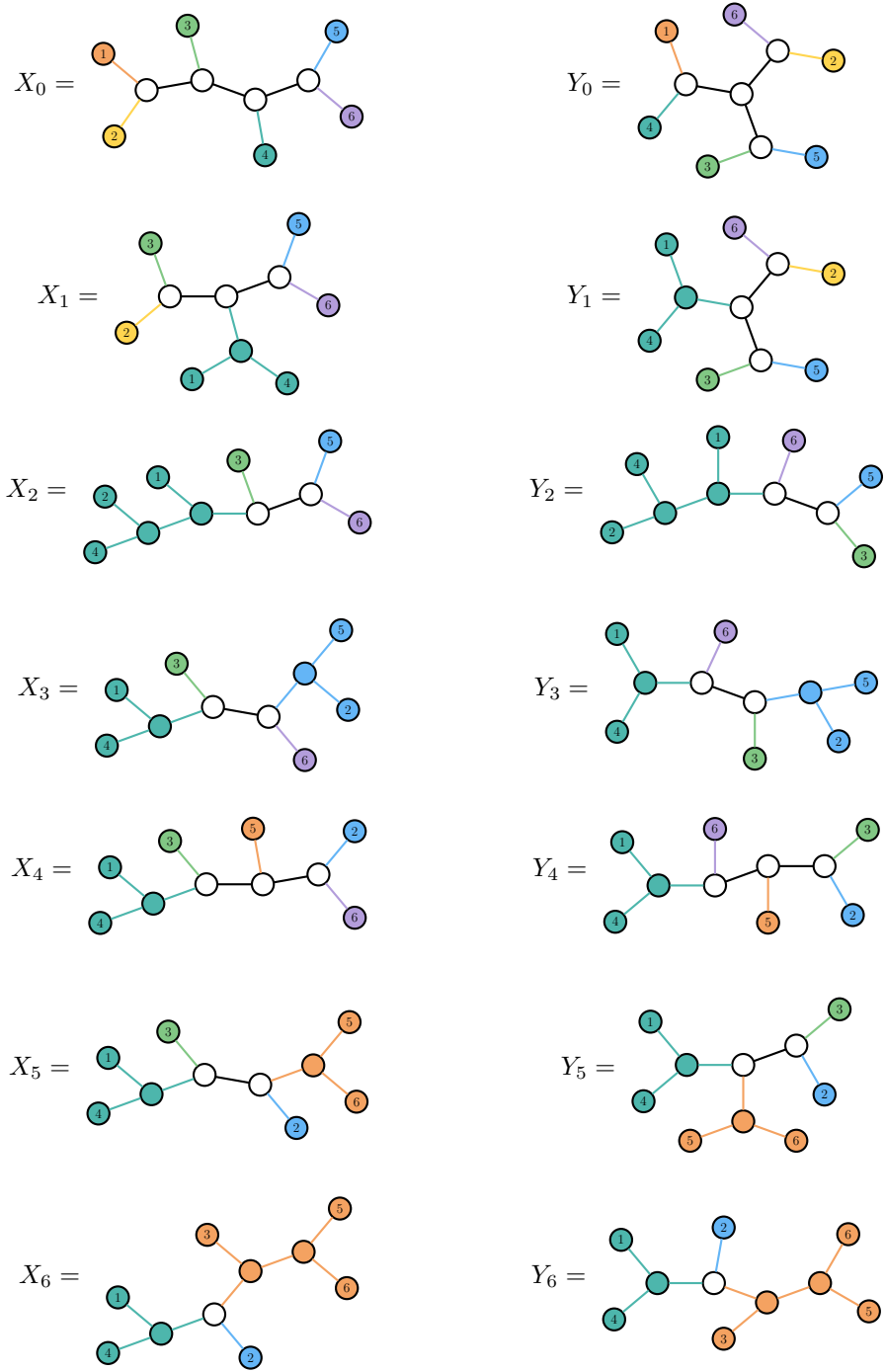

We now describe the dynamics of the coupling, which we illustrate in Figure~\ref{fig: coupling illustration}. 
As long as $X_t$ (and hence also $Y_t$) has at least 4 colored components,
we apply a colored step to both $X_t$ and $Y_t$ as follows.
In the first half-step, we pick a single label uniformly at random in $\{1,\dots,n\}$,
and we remove the corresponding leaf (with its incident edge) from $X_t$ and $Y_t$;
then in the second half-step, we reinsert the removed leaf on a uniform random edge $ \ag v_1, v_2 \ad$ of $X_t$ and, 
\begin{itemize}
    \item if $ \ag v_1, v_2 \ad$ has a color $i>0$, then 
    we reinsert the removed leaf (and its incident component) in $Y_t'$ \textbf{on the same edge} of the same colored component as in $X_t'$
    (since $\cP_i(X_t')=\cP_i(Y_t')$, we can identify their edges, and reinserting it on the same edge in $Y_t'$
    as in $X_t'$ makes sense);
    \item if $ \ag v_1, v_2 \ad$ is uncolored, then we choose an uncolored edge of $Y_t'$ uniformly at random,
    and we reinsert the removed leaf on it.
\end{itemize}
This gives $X_{t+1}$ and $Y_{t+1}$, and in both cases, our invariant
is preserved, i.e.~$\cP_i(X_{t+1})=\cP_i(Y_{t+1})$ for all $i>0$. 
One can also see that both $X_t$ and $Y_t$ individually are Markov chains with transitions as described in Section~\ref{s: colored steps}.
For $X_t$, this is immediate. For $Y_t$, we need to verify that the edge on which we reinsert the removed leaf is chosen uniformly at random. But, by the first item, edges in colored components have the same probability to be chosen in $Y'_t$ as in $X'_t$, i.e.~$1/(2n-5)$. By construction, in the second item, edges outside colored components all have the same probability to be chosen, which therefore has to be $1/(2n-5)$ as well. Hence, $(Y_t)$ is also a Markov chain with the same transition probabilities as $(X_t)$, and we have indeed constructed a coupling of two copies of our colored Markov chain.

\medskip

Let us define the stopping time
\begin{equation}
    T = \min\ag t\geq 0 \mid X_t \text{ has } 3 \text{ colored components}\ad.
\end{equation}
The crucial point is that $X_T = Y_T$. Indeed, both $X_T$ and $Y_T$ consist of the same three colored components  attached to a single uncolored vertex. Since cladograms are not embedded in the plane, we have $X_T=Y_T$.
Therefore, for $t \ge T$, we can forget colors and
apply the dynamics of Section~\ref{s: the model} to $X_t$ and $Y_t$ simultaneously, preserving the equality $X_t=Y_t$.

\medskip

Then, from standard coupling arguments -- see, e.g., \cite[Corollary 5.3]{LivreLevinPeres2019MarkovChainsAndMixingTimesSecondEdition} --, one has the following relation between
the coupling time $T$ and the worst case total variation distance $ d^{(n)}(t)$ defined in~Eq.~\eqref{eq: def d n t}.
\begin{lemma}\label{lem:dtv-tailT}
    $\displaystyle d^{(n)}(t) \le \sup_{x,y \in \cC_n} \bbP \big[ T >t \big| X_0 = \AddColors(x), \, Y_0 = \AddColors(y)\big]$.
\end{lemma}
In the next section,
we prove that the distribution of the stopping time does not depend on the initial cladograms $x$ and $y$,
and provide some bounds on its tail probabilities.

\section{Analysis of the coupling time}

\subsection{Projection on component sizes}
Recall that, for a colored cladogram $X$, we denote by $\cP_i(X)$ its component of color $i$.
We also let $F_i(X)$ be the number of leaves of color $i$ in $X$. 
The number of interior vertices and edges of color $i$ in $X$ are then $F_i(X)-1$ and $2F_i(X)-1$, given that $F_i(X)>0$.

Let $(X_t)_{t \ge 0}$ be the Markov chain on colored cladograms defined in Section~\ref{s: colored steps}.
Then it is easy to see that the vector $(F_1(X_t),\dots,F_n(X_t))_{t \ge 0}$ is also a Markov chain with
initial condition $(1,\dots,1)$ and transitions consisting in
\begin{itemize}
    \item a first half-step
    changing $(F_1,\dots,F_n)$ to $(F'_1,\dots,F'_n)\coloneq (F_1,\dots,F_i-1,\dots,F_n)$ with probability $F_i/n$ (for $1\le i \le n$);
    \item then a second half-step changing $(F'_1,\dots,F'_n)$ to 
    $(F'_1,\dots,F'_j+1,\dots,F'_n)$ with probability $(2F'_j-1)/(2n-5)$
    and changing the first $0$ of $(F'_1,\dots,F'_n)$ into a 1 with probability $(C'-3)/(2n-5)$, where $C' = |\{j\ge 1:\,F'_j \ne 0\}|$.
\end{itemize}
In particular the distribution of $F_1(X_t),\dots,F_n(X_t)$, and hence that of the stopping time $T$,
is independent of the initial cladograms $x,y \in \cC_n$.

\begin{remark}\label{rmk:partition_chain_in_Petrov}
    Reordering the vector $(F_1,\dots,F_n)$ in nonincreasing order yields a Markov chain on {\em integer partitions} of $n$.
    This Markov chain is a down-up chain similar to those studied by Petrov\footnote{Petrov studied up-down chains rather than down-up chains, in the sense that, the two half-steps are interchanged in his construction compared to what we are doing here.
    This does not change the nature of the results.} in \cite{petrov2009diffusions-Kingman}
    with parameters $\theta =-3/2$ and $\alpha=1/2$
    (see also \cite[Section 5.3]{FerayRiveraLopez2025UpDownChainsAndScalingLimits}).
    The results of these papers however do not apply directly since they assume $\alpha+\theta>0$,
    which is a necessary condition for the chain to be irreducible, while here we are in a parameter regime
    where we cannot create new nonempty parts when we have only three of them.
    Nevertheless, it might be possible to find some eigenvector/eigenvalue decomposition for the transition operator
    of the chain and to compute the distribution of the stopping time $T$.
    We shall follow a more direct path here, and bound the tail probability $\bbP[T \ge t]$
    by introducing a relevant statistic.
\end{remark}

\subsection{A useful statistic with a drift}\label{s: a useful statistic}
We define the following statistic on colored cladograms: for $X\in \widetilde{\cC}_n$, we set:
\begin{equation}
    \varphi(X) = \sum_{i=1}^n F_i(X)^2,
\end{equation}
where we emphasize that the sum starts at $i=1$.
The key observation, proved below, is that applying a single color step 
increases the value of this statistic by at least $1$ on average.
Heuristically, since the statistic cannot exceed $n^2$, this implies that the coupling
must stop after roughly $n^2$ steps, i.e.~$T=O(n^2)$ (in probability).
This implication will be made rigorous in Section~\ref{s: bounding the coupling time}.
\medskip

Let $X\in \widetilde{\cC}_n$ be a (nonrandom) colored cladogram with $n$ leaves.
\textbf{We assume that $X$ contains at least four colors.}
We denote the colored cladogram obtained from $X$ after the first colored half-step by $X'$, and the colored cladogram obtained after the second half step by $X''$, as defined in Section~\ref{s: colored steps}.

\begin{lemma}\label{lem: expected drift}
We have $\bbE\cg \varphi(X'')- \varphi(X)\cd \geq 1$.
\end{lemma}
\begin{proof}
Throughout the proof we write $F_i=F_i(X)$ and $F'_i=F_i(X')$ to lighten the notation.
    We first compute $\bbE\cg \varphi(X')- \varphi(X)\cd$. For each $i \in \{1,\dots,n\}$, the probability 
    that the uniformly chosen leaf $\ell$ has color $i$ is $F_i/n$. In this case, one has 
    \[\varphi(X')- \varphi(X) = (F_i-1)^2 - F_i^2=-2 F_i + 1. \]
    Hence we have
    \begin{equation}\label{eq:phiX'-phiX}
        \bbE\cg \varphi(X')- \varphi(X)\cd = \sum_{i=1}^n \tfrac{F_i}n \, (-2 F_i + 1) = -\frac{2\varphi(X)}{n} + 1,
    \end{equation}
    where we used that $\sum_{i=1}^n F_i=n$.

    We now study $\varphi(X'')- \varphi(X')$, conditionally on $X'$. For $i \le n$, recall that $F'_i$ is the number of
    leaves of color $i$ in $X'$. Note that $\sum_{i=1}^n F'_i=n-1$ since $X'$ has $n-1$ leaves.
    When $F'_i>0$, the number of edges with color $i$ in $X'$ is then $2F'_i-1$,
    while the total number of edges of $X'$ is $2n-5$. Hence, the probability to re-attach the removed leaf on an edge of color $i$ is $(2F'_i-1)/(2n-5)$.
    In this case, we have
    \[\varphi(X'')- \varphi(X') = (F'_i+1)^2-F_i'^2 = 2F'_i+1.\]
    In the remaining cases, i.e.~when the removed leaf is attached to an uncolored edge,
    attaching the new edge creates a new component of size $1$ and we have
    \[\varphi(X'')- \varphi(X') =1.\]
    Hence, we get
    \begin{equation} \label{eq:phiX''-phiX'}
        \bbE\cg \varphi(X'')- \varphi(X')|X'\cd = 1 + \sum_{i : F'_i>0} \frac{2F'_i-1}{2n-5}\, (2F'_i) = 1+\frac{4 \varphi(X')}{2n-5} - \frac{2(n-1)}{2n-5}.
    \end{equation}
The rest is a direct computation. Adding up \eqref{eq:phiX'-phiX} and the expectation of \eqref{eq:phiX''-phiX'}, we get
\[\bbE\cg \varphi(X'')- \varphi(X)\cd = 2 - \frac{2(n-1)}{2n-5} + \frac{4 \bbE \cg \varphi(X') \cd}{2n-5}  -\frac{2\varphi(X)}{n}.\]
Using \eqref{eq:phiX'-phiX} again, we obtain
\begin{equation}
    \bbE\cg \varphi(X'')- \varphi(X)\cd = 2 - \frac{2(n-1)}{2n-5} +\frac{4}{2n-5} +\varphi(X)\left(
 \frac{4}{2n-5} - \frac{8}{n(2n-5)} - \frac2n \right).
\end{equation}
After simplification, this gives us
\begin{equation} \label{eq:drift} 
\bbE\cg \varphi(X'')- \varphi(X)\cd=\frac{2n-4}{2n-5} + \varphi(X) \frac{2}{n(2n-5)} \ge 1,
 \end{equation}
 which concludes the proof.
\end{proof}
\begin{remark}
    Interestingly, an asymptotic version of Eq.~\eqref{eq:drift} has previously appeared 
    in~\cite{petrov2009diffusions-Kingman} (although in a different parameter regime; recall Remark~\ref{rmk:partition_chain_in_Petrov}).
    Indeed, let us define the  renormalization 
    $\tilde \varphi_n(t)=n^{-2} \varphi(X_{\lfloor n^2 t\rfloor})$,
    and assume that $\tilde \varphi_n(t)$ converges to some continuous function $\tilde \varphi(t)$ for large $n$. 
    Then Eq.~\eqref{eq:drift}, applied $\lf n^2 \dt \rf$ times, implies that
    \[\bbE[\tilde \varphi(t+\dt) -\tilde \varphi(t) | \mathcal F_t]
    = \dt(1 + \tilde \varphi(t)) +o(\dt),\]
    where $\mathcal F_t$ is the $\sigma$-algebra generated by $(\tilde \varphi(u))_{u \le t}$.
    Translated in the formalism of generators, this corresponds to \cite[Eq.~(1)]{petrov2009diffusions-Kingman}, specialized to $\theta =-3/2$ and $\alpha=1/2$, and applied to $q_1$. This asymptotic version is, however, not sufficient for our purposes.
\end{remark}

\subsection{Bounding the coupling time}\label{s: bounding the coupling time}

We now want to bound the tail probabilities of the stopping time $T$.
Let $n\geq 4$, $x\in \cC_n$ and set $X_0 = \AddColors(x)$, as in the initial configuration of the coupling. 
For $t\geq 0$, let $X_{t+1}$ be the (random) colored configuration obtained by applying one colored step to $X_t$. Recall that $T$ is the first time at which $X_t$ has 3 colors. In particular, for $t<T$, $X_t$ has at least $4$ colors, so we can apply the drift estimate from Lemma~\ref{lem: expected drift}.
\medskip

Since $\varphi(X_t)$ cannot exceed $n^2$ and $\varphi(X_T)$ is of order $n^2$,
the stopping time $T$ looks roughly like the exit time from the domain $[0,n^2]$ of the biased random walk $(\varphi(X_t))_{t\geq 0}$ (note however that the process $(\varphi(X_t))_{t\geq 0}$ is not Markovian and that the stopping time $T$ cannot be defined using $\varphi(X_t)$ only).
Much is known about such exit times, see e.g.~\cite{Foster1953} or~\cite[Chapter 11]{MeynTweedie2009}. 
In particular, the following lemma uses standard arguments; we provide a proof for completeness.
\begin{lemma}\label{lem: drift into proba}
    Let $\beta\geq 1$ be an integer. Then $\bbP(T > \beta\, n^2) \leq 1/\beta$.
\end{lemma}
\begin{proof}
  Set $m=\beta\, n^2$. Let $0\leq t\leq m$. By monotonicity, we have $\bbP(T>m)\leq \bbP(T>t)$. Also, we have $\bbE\cg \varphi(X_{t+1}) - \varphi(X_t) \mid T>t \cd\geq 1$ by Lemma~\ref{lem: expected drift}. Multiplying these two inequalities yields
    \begin{equation}\label{eq: eq inter in lem: drift into proba for each t}
       \bbP(T>m) \leq  \bbP(T>t)\, \bbE\cg \varphi(X_{t+1}) - \varphi(X_t) \mid T>t \cd = \bbE \cg \varphi(X_{(t+1)\wedge T}) - \varphi(X_{t\wedge T})\cd.
    \end{equation}
Averaging \eqref{eq: eq inter in lem: drift into proba for each t} for $t$ from $0$ to $m-1$, and using that $0\leq \varphi(x) \leq n^2$ for any $x\in \widetilde{\cC}_n$, we obtain that $\bbP(T>m) \leq \bbE\cg \varphi(X_{m\wedge T}) - \varphi(X_0)\cd/m \leq n^2/m = 1/\beta$, which concludes the proof.
\end{proof}

\subsection{Proof of Theorem~\ref{thm: main mixing time cladograms}}
By Lemmas~\ref{lem:dtv-tailT} and~\ref{lem: drift into proba}, for any integer $\beta\geq 1$, we have
\begin{equation}
    \textup{d}^{(n)}(\beta n^2) \leq \bbP(T>\beta n^2) \leq 1/\beta,
\end{equation}
where $\textup{d}^{(n)}(\cdot)$ was defined in \eqref{eq: def d n t}.
This implies, taking $\beta = 4$, that $t_{\textup{mix}}^{(n)} = O(n^2)$. Since the lower bound was proved in \cite[Section 2]{Aldous2000MixingCladograms}, the proof of Theorem~\ref{thm: main mixing time cladograms} is complete.

\section*{Acknowledgements}
This work was supported by the ANR project LOUCCOUM (ANR-24-CE40-7809).

\section*{Use of AI}
The key arguments and proofs were developed and written by the authors. 

ChatGPT 5.6 Sol was used for the exit time argument in Section~\ref{s: bounding the coupling time}, proofreading, and to help with figure generation. For the exit time argument, ChatGPT told us that Lemma~\ref{lem: drift into proba} follows directly from a statement in \cite{MeynTweedie2009}, and we asked it to extract the necessary arguments for our case. This led to the short proof of Lemma~\ref{lem: drift into proba}, which we decided to include.

The authors take full responsibility for the content of the paper.

\footnotesize{
\bibliography{bibliographieLucas}
\bibliographystyle{alpha}}

\end{document}